\documentclass[12pt]{amsart}

\usepackage[margin=1.2in]{geometry}

\usepackage{amssymb,amsfonts}
\usepackage[utf8]{inputenc}
\usepackage{tikz,tikz-cd}

\usepackage[colorinlistoftodos]{todonotes}
\usepackage{ytableau}
\usepackage{verbatim}

\usepackage{hyperref}

\title{A Row Analogue
of Hecke Column Insertion}

\author{Daoji Huang}

\author{Mark Shimozono}

\author{Tianyi Yu}

\theoremstyle{plain}
\newtheorem{theorem}{Theorem}[section]
\newtheorem{lemma}[theorem]{Lemma}
\newtheorem{corollary}[theorem]{Corollary}
\newtheorem{proposition}[theorem]{Proposition}

\theoremstyle{definition}
\newtheorem{definition}[theorem]{Definition}
\newtheorem{example}[theorem]{Example}

\theoremstyle{remark}
\newtheorem{remark}[theorem]{Remark}

\numberwithin{equation}{section}
\numberwithin{figure}{section}
\numberwithin{table}{section}

\newcommand{\len}{\mathrm{len}}
\newcommand{\Red}{\mathrm{Red}}
\newcommand{\Dec}{\mathsf{Dec}}
\newcommand{\Inc}{\mathsf{Inc}}
\newcommand{\SSYT}{\mathsf{SSYT}}
\newcommand{\RSSYT}{\mathsf{RSSYT}}
\newcommand{\SVT}{\mathsf{SVT}}
\newcommand{\RSVT}{\mathsf{RSVT}}
\newcommand{\CP}{\mathsf{CP}}

\newcommand{\la}{\lambda}

\newcommand{\wt}{\mathrm{wt}}
\newcommand{\Y}{\mathbb{Y}}
\newcommand{\Z}{\mathbb{Z}}

\newcommand{\row}{\mathrm{row}}
\newcommand{\shape}{\mathrm{shape}}

\newcommand{\Hequiv}{\equiv_H}
\newcommand{\IS}{\Inc_w \times_\Y \SVT}
\newcommand{\IR}{\Inc_w \times_\Y \RSVT}
\newcommand{\DS}{\Dec_w \times_\Y \SVT}
\newcommand{\DR}{\Dec_w \times_\Y \RSVT}
\newcommand{\ISred}{\Inc_w^\Red \times_\Y \SSYT}
\newcommand{\IRred}{\Inc_w^\Red \times_\Y \RSSYT}
\newcommand{\DSred}{\Dec_w^\Red \times_\Y \SSYT}
\newcommand{\DRred}{\Dec_w^\Red \times_\Y \RSSYT}

\newcommand{\bIS}{\Phi_{IS}}
\newcommand{\bIR}{\Phi_{IR}}
\newcommand{\bDS}{\Phi_{DS}}
\newcommand{\bDR}{\Phi_{DR}}
\newcommand{\ev}{\mathrm{ev}}
\newcommand{\rev}{\mathrm{rev}}

\newcommand{\PhiRSK}{\widetilde{\Phi}}

\definecolor{myGreen}{rgb}{0.02, 0.6, 0.02}

\begin{document}

\maketitle


\begin{abstract}
 We introduce a new row insertion algorithm 
on decreasing tableaux and increasing tableaux, generalizing Edelman-Greene (EG) row insertion. 
Our row insertion algorithm is a nontrivial variation of Hecke column insertion which generalizes EG column insertion.     
Similar to Hecke column insertion,
our row insertion is bijective and respects Hecke equivalence, and
therefore recovers the expansions of
stable Grothendieck functions into 
Grassmannian stable Grothendieck functions. 
\end{abstract}

\section{Introduction}
Edelman-Greene (EG) insertion was introduced to give combinatorial expansions of Stanley symmetric functions into Schur functions \cite{EG}. EG insertion achieves this expansion because it respects Coxeter-Knuth (nilplactic) equivalence (an equivalence relation on reduced words for permutations) and satisfies a Pieri rule which guarantees that the recording tableau is semistandard. 
EG insertion comes in four flavors depending on 
the use of increasing versus decreasing tableaux and row versus column insertion. 
The four flavors are essentially the same: the row and column versions are related by the naive transpose and the increasing and decreasing versions are related by reversing the total order on entries.
For particular applications one might require a specific flavor.
For the Stanley-to-Schur expansions both EG column insertion and EG row insertion can be used.
But there is a subtlety here. To get a semistandard recording tableau for EG column insertion, one must use a certain kind of biword as input. Using the same biword for EG row insertion results in a recording tableau which is the \textit{transpose} of semistandard. One must use a different kind of input biword.
The transformation between the two kinds of input biwords involves reversing the reduced word.

Hecke column insertion was introduced in \cite{BKSTY} to give combinatorial expansions of stable Grothendieck functions $G_w$ into stable Grassmannian Grothendieck functions $G_\lambda$. Hecke column insertion realizes these expansions because it respects Hecke equivalence and satisfies a Pieri property which guarantees that the recording tableaux be set-valued. The increasing and decreasing versions of Hecke column insertion
generalize the two flavors of EG column insertion. 

Let us consider the problem of generalizing the two kinds (increasing/decreasing tableaux) of EG row insertion while respecting Hecke equivalence.
The naive transpose of Hecke column insertion respects Hecke equivalence and directly generalizes EG row insertion but satisfies a Pieri property which implies that the recording tableau is the 
\textit{transpose} of a set-valued tableau. This also cannot be fixed by transforming the input biword; operations such as reversal will not even recognizably transform the shape of the output tableau.

Our new insertion generalizes EG row insertion, respects Hecke equivalence and also satisfies the Pieri property which produces set-valued tableaux (as opposed to the transpose of set-valued tableaux). 
This novel insertion has the very unusual property that some values may be moved which are not part of the bumping path. 
One application of this insertion is the expansion
of $G_w$ times a Lascoux polynomial into Lascoux polynomials, which is not achievable by Hecke column insertion \cite{OY}.

\subsection{Various functions}
The main application of these various insertion algorithms, is to expand
the Stanley symmetric functions $F_w$ and stable Grothendieck functions $G_w$.
We define them combinatorially using words. 
A pair of words $(a_1 \cdots a_n,i_1 \cdots i_n)$ is called \textit{compatible}\footnote{These compatible sequences are the reverse words of those defined in \cite{BJS}.}
if $i_j \geq i_{j+1}$ and $i_j = i_{j+1}$
implies $a_j < a_{j+1}$ for all $1 \leq j < n$.
Each word $a$ has an associated permutation $[a]_H$ (see \S \ref{SS:Hecke}).
We say $a$ is a \textit{Hecke word} for $w$ if $[a]_H = w$.
Let $\CP_w$ be the set of compatible pairs $(a, i)$
such that $a$ is a Hecke word for $w$. 
Let $\CP_w^{\Red}$ consist of $(a,i) \in \CP_w$
such that $a$ is reduced (i.e. $\len(a) = \ell(w)$ where $\ell(w)$ is the 
Coxeter length).
By~\cite{BJS} and~\cite{FK},
\begin{align*}
F_w &= \sum_{(a,i) \in \CP_w^\Red} x^{\wt(i)}  \\
G_w &= \sum_{(a,i) \in \CP_w} (-1)^{|\wt(i)| - \len(w)}x^{\wt(i)} \:.
\end{align*}

We define the Schur function $s_\lambda$ and the stable Grassmannian Grothendieck function $G_\lambda$ using tableaux. 
Let $\Y$ be the set of partitions. 
For $\la=(\la_1\ge\la_2\ge\dotsm)\in\Y$,
let $D(\la) = \{(i,j)\in\Z_{>0}^2\mid 1\le j\le \la_i\}$ be its diagram under the English convention with matrix-style indexing. A set-valued tableau $T$ of shape $\la\in\Y$ is a function which assigns to each $s\in D(\la)$ a nonempty finite subset of $\mathbb{Z}_{>0}$, such that 
if $s'$ is immediately to the right (resp. below) $s$ in the same row (resp. column) then $\max(T(s)) \le \min(T(s'))$ (resp. $\max(T(s)) < \min(T(s'))$).
We denote by $\SVT$ (resp. $\RSVT$) the set of set-valued tableaux (resp. reverse set-valued tableaux, meaning all inequalities are reversed). Let $\SSYT$ (resp. $\RSSYT$) denote the set of semistandard (resp. reverse semistandard) Young tableaux, meaning set-valued (resp. reverse set-valued) tableaux in which each set is a singleton. The following is a reverse-set-valued tableau of shape $(3,2)$.
\ytableausetup{boxsize=9mm}
\[
\begin{ytableau}
   5 & 5,4 & \scalebox{0.9}{3,2,1} \\ 3,2 & 2
\end{ytableau}
\]
\ytableausetup{boxsize=normal}


Then $s_\la$ and $G_\la$ each have two equivalent formulas (\cite{Bu} for $G_\la$): 
\begin{align*}
s_\la &= \sum_{\substack{Q \in \SSYT \\ \shape(Q) = \la}} x^{\wt(Q)} 
= \sum_{\substack{Q \in \RSSYT \\ \shape(Q) = \la}} x^{\wt(Q)} \\
G_\la &= \sum_{\substack{Q \in \SVT \\ \shape(Q) = \la}} (-1)^{|\wt(Q)| - |\la|}x^{\wt(Q)} 
= \sum_{\substack{Q \in \RSVT \\ \shape(Q) = \la}} (-1)^{|\wt(Q)| - |\la|}x^{\wt(Q)}\:.
\end{align*}
 
\subsection{Expansions}
The $F_w$ (resp. $G_w$) can be expanded into $s_\la$ 
(resp. $G_\la$).
The expansion coefficients have geometric meaning; they contain all cohomological (resp. $K$-theoretic) equioriented type A quiver constants as special cases \cite{BKSTY}.

There are two ways to write down either of the
two expansions, 
using either increasing tableaux 
or decreasing tableaux.
For a permutation $w$,
let $\Inc_w$ (resp. $\Dec_w$) be the set of increasing (resp. decreasing) tableaux $P$ whose row word 
$\row(P)$ (resp. reverse row word $\rev(\row(P))$; see \S \ref{SS:Hecke}) is a Hecke word 
for $w$. Let $\Inc_w^\Red$ (resp. $\Dec_w^\Red$) 
consists of tableaux in $\Inc_w$ (resp. $\Dec_w$)
whose row word is reduced.
Then we have (\cite{EG} for $F_w$ and \cite{BKSTY} for $G_w$)
\begin{align}
\label{E:F Inc}
F_w &= \sum_{P \in \Inc_w^\Red} s_{\shape(P)} \\
\label{E:F Dec}
&= \sum_{P \in \Dec_w^\Red} s_{\shape(P)}, \\
\label{E:G Inc}
G_w & = \sum_{P \in \Inc_w} (-1)^{\ell(w) - |\shape(P)|} G_{\shape(P)} \\
\label{E:G Dec}
& = \sum_{P \in \Dec_w} (-1)^{\ell(w) - |\shape(P)|} G_{\shape(P)} .  
\end{align}

\subsection{Insertion algorithms: General requirements}
\label{SS:general insertion}
Let $A$ and $B$ be sets of tableaux of partition shape.
We use the notation
\begin{align}
  A \times_\Y B = \{ (P,Q)\in A\times B \mid \shape(P)=\shape(Q)\}
\end{align}
for the fiber product over the maps $A\to \Y$ and $B\to \Y$  given by taking the shape of a tableau.

To give a combinatorial proof of \eqref{E:G Inc} it suffices to produce a bijection $\bIS:\CP_w\to \IS$ or $\bIR: \CP_w\to \IR$ which is weight-preserving:
\begin{align}
  (a,i)&\mapsto(P,Q) \\
  \wt(i) &= \wt(Q).
\end{align}
Similarly to prove \eqref{E:G Dec} it suffices to supply a
weight-preserving bijection $\bDS:\CP_w\to \DS$ or $\bDR:\CP_w\to \DR$.

\subsection{Edelman-Greene insertion: solution for reduced case}
Historically first to be discovered were 
``reduced" restrictions of the above 
bijections. 
The expansions \eqref{E:F Inc} and \eqref{E:F Dec} are 
obtained via four weight-preserving bijections.
These bijections are given by four variations of the
Edelman-Greene insertion (EG insertion) \cite{EG}:
\begin{itemize}
\item $\bIS^\Red:\CP_w^\Red \rightarrow \ISred$: 
EG column insertion into increasing tableaux, 
starting from the right end of the compatible pairs.
\item $\bIR^\Red: \CP_w^\Red \rightarrow \IRred$: 
EG row insertion into increasing tableaux, 
starting from the left end of the compatible pairs.
\item $\bDS^\Red: \CP_w^\Red \rightarrow \DSred$: 
EG row insertion into decreasing tableaux, 
starting from the right end of the compatible pairs.
\item $\bDR^\Red: \CP_w^\Red \rightarrow \DRred$: 
EG column insertion into decreasing tableaux, 
starting from the left end of the compatible pairs.
\end{itemize}
The four ``reduced" bijections are essentially equivalent: 
EG row insertion and EG column insertion are merely 
transposes of each other. 
The relationships are summarized in the commutative diagram in Figure~\ref{F:reduced bijections}.

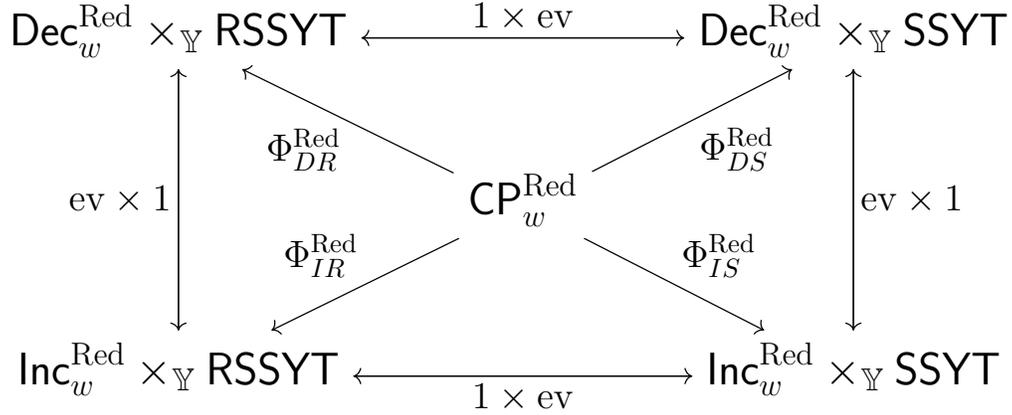
\begin{figure}
\[
\tikzcdset{every label/.append style = {font = \large}, nodes={font=\Large}}
\begin{tikzcd}
[row sep=3em, column sep = 3em]
\Large{\DRred} \arrow[rr, "1 \times \ev"] \arrow[dd, swap, "\ev \times 1"]  &&
\DSred \arrow[ll] \arrow[dd, "\ev \times 1"]\\
& \CP_w^\Red \arrow[lu,"\bDR^\Red"]\arrow[ru,swap,"\bDS^\Red"]\arrow[ld,swap,"\bIR^\Red"]\arrow[rd,"\bIS^\Red"]\\
\IRred \arrow[rr, swap, "1 \times \ev" ] \arrow[uu]&&
\ISred \arrow[ll] \arrow[uu]
\end{tikzcd}
\]
\caption{The four reduced bijections}
\label{F:reduced bijections}
\end{figure}

\begin{example} Consider the following element $(a,i)\in \CP^\Red$.
\begin{align*}
\begin{array}{|c||cccccccc|} \hline 
i & 3&3&3&2&2&2&1&1\\
a & 1&2&4&1&3&5&2&4 \\ \hline
\end{array}
\end{align*}
We have
\ytableausetup{boxsize=5mm,aligntableaux=top}
\begin{align*}
\bDR(a,i) = \left(\, \raisebox{0.3cm}{\ytableaushort{5421,431,2}}\,,\,\raisebox{0.3cm}{\ytableaushort{3332,221,1}}\,\right) 
\bDS(a,i) = \left(\, \raisebox{0.3cm}{\ytableaushort{5421,431,2}}\,,\,\raisebox{0.3cm}{\ytableaushort{1123,223,3}}\,\right) \\
\bIR(a,i) = \left(\, \raisebox{0.3cm}{\ytableaushort{1234,235,4}}\,,\,\raisebox{0.3cm}{\ytableaushort{3332,221,1}}\,\right) 
\bIS(a,i) = \left(\, \raisebox{0.3cm}{\ytableaushort{1234,235,4}}\,,\,\raisebox{0.3cm}{\ytableaushort{1123,223,3}}\,\right) \\
\end{align*}

There are two kinds of evacuation maps
in Figure~\ref{F:reduced bijections}.
The map $\ev:\SSYT\to\RSSYT$ is defined as follows. 
For a $T \in \SSYT$, there is a unique 
$T' \in \RSSYT$ such that $\row(T)$
and $\rev(\row(T'))$ are Knuth equivalent,
where $\rev(\cdot)$ is the operator that reverses a word.
Then $\ev(T) := T'$.
The computation $\ev:\SSYT\to\RSSYT$ 
can be done by jeu-de-taquin as follows.
Sliding out the 1's using the usual 
jeu-de-taquin we obtain
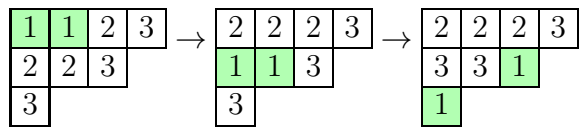
\begin{align*}
\begin{ytableau}
  *(green!30)1 &*(green!30) 1 &2&3\\
  2&2&3 \\
  3
\end{ytableau}\to 
\begin{ytableau}
  2&2&2&3\\
*(green!30) 1 & *(green!30) 1&3 \\
3 
\end{ytableau}\to 
\begin{ytableau}
  2&2&2&3\\
3&3& *(green!30) 1 \\
*(green!30) 1
\end{ytableau}
\end{align*}
Then the 2's are slid out but not past the 1s.
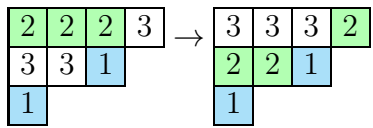
\begin{align*}
\begin{ytableau}
 *(green!30) 2&*(green!30) 2&*(green!30) 2&3\\
3&3& *(cyan!30) 1 \\
*(cyan!30) 1
\end{ytableau}\to
\begin{ytableau}
 3&3&3&*(green!30)2\\
*(green!30)2&*(green!30) 2& *(cyan!30) 1 \\
*(cyan!30) 1
\end{ytableau}
\end{align*}
The 3's need no moving. The result is
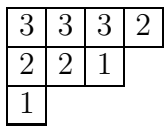
\begin{align*}
  \ytableaushort{3332,221,1}
\end{align*}
\end{example}

The following Proposition asserts that the lower triangle in Figure \ref{F:reduced bijections} commutes.

\begin{proposition}[{\cite[Corollary. 7.22]{EG}}]
\label{P:EG left right}
Let $(a,i)\in\CP^\Red_w$ and $\bIS^\Red(a,i)=(P,Q)$ and $\bIR^\Red(a,i)=(P',Q')$. Then
$P=P'$ and $Q'=\ev(Q)$ where $\ev: \SSYT\to\RSSYT$ is Sch\"utzenberger's 
evacuation involution (usual evacuation but without relabeling).
\end{proposition}

The upper triangle also commutes: it is the same statement but with the total order on values 
reversed. 

The other kind of evacuation map $\ev:\Inc^\Red_w\to\Dec^\Red_w$ 
can be defined similarly. 
For a $T \in \Inc^\Red_w$, there is a unique 
$T' \in \Dec^\Red_w$ such that $\row(T)$
and $\rev(\row(T'))$ are Coxeter-Knuth equivalent~\cite{EG}.
Then $\ev(T) := T'$.
The following result says that the triangle on the right of Figure \ref{F:reduced bijections} is commutative.

\begin{proposition} \label{P:EG inc dec}
Let $(a,i)\in\CP^\Red_w$, $\bIS^\Red(a,i)=(P,Q)$ and $\bDS^\Red(a,i)=(P',Q')$. 
Then $Q=Q'$ 
and $P'=\ev(P)$.
\end{proposition}

\begin{proof}
The statement for $Q$ tableaux is proved in \cite[Corollary. 7.21]{EG}.
By~\cite[Theorem. 6.24]{EG},
$\row(P)$ and $a$ are Coxeter-Knuth equivalent. 
On the other hand, 
$\row(P')$ and $\rev(a)$ are Coxeter-Knuth equivalent.
Thus, $\rev(\row(P'))$ and $\row(P)$ are Coxeter-Knuth equivalent,
so $P' = \ev(P)$.
\end{proof}

Similarly, the triangle on the left also commutes. 
Thus, Figure~\ref{F:reduced bijections} commutes. 
In particular, for any fixed $(a,i)\in \CP_w^\Red$, upon applying 
any of the four EG bijections, the tableau pair has the same shape.

\subsection{Solutions for general case}

Hecke column insertion \cite{BKSTY} defines a bijection
$\bIS: \CP_w \rightarrow \IS_w$ whose restriction to $\CP^\Red_w$ is
EG column insertion; the insertion starts at the left end of the word. By merely reversing the total order on entries in tableaux
and inserting starting from the right end of the input word,
the resulting variant of Hecke column insertion 
gives a bijection $\bDR: \CP_w \rightarrow \DR_w$.

However, there are no known easy variations 
of the Hecke insertion which achieve the bijections
$\bIR$ or $\bDS$. 
The Hecke row insertion is the variant of the Hecke column insertion 
in which the roles of rows and columns are exchanged. A slightly restricted version of the row Hecke insertion was considered
by Patrias and Pylyavskyy~\cite{PP}, where they studied the growth diagrams of Hecke row insertion of a $(a,i)$ where $i$ has distinct entries. In this case, the recording tableaux are valid SVTs.
Unfortunately, when the input is an arbitrary compatible pair, the row Hecke insertion does not always produce a valid SVT as the recording tableau. We illustrate this failure with an example. 

\begin{example}[Pathology of Hecke row insertion]
    \label{ex:hecke-row-fail}
    Consider the compatible pair
    \begin{align*}
\begin{array}{|c||ccccccc|} \hline 
i & 3&3&2&2&2&1&1\\
a & 1&3&1&2&3&1&3\\ \hline
\end{array}
\end{align*}
The procedure for Hecke column insertion $\bIS:\CP_w\to \IS$ is
\begin{align*}
    &\left(\,
 \raisebox{-0.1cm}
{\ytableaushort{3}}\,,\,
\raisebox{-0.1cm}
{\ytableaushort{1}}\,\right)\to
\left(\,
\raisebox{-0.1cm}
{\ytableaushort{13}}\,,\,
\raisebox{-0.1cm}
{\ytableaushort{11}}\,\right)\to
\left(\,
\raisebox{0.1cm}
{\ytableaushort{13,3}}\,,\,
\raisebox{0.1cm}
{\ytableaushort{11,2}}\,\right) \to
\left(\,
\raisebox{0.1cm}
{\ytableaushort{13,2}}\,,\,
\raisebox{0.1cm}{
\begin{ytableau}
 1&\scalebox{0.8}{1,2}\\
 2
\end{ytableau}}\,\right) \\
&\to
\left(\,
\raisebox{0.1cm}
{\ytableaushort{123,2}}\,,\,
\raisebox{0.1cm}{
\begin{ytableau}
 1&\scalebox{0.8}{1,2}&2\\
 2
\end{ytableau}}\,\right)\to
\left(\,
\raisebox{0.3cm}
{\ytableaushort{123,2,3}}\,,\,
\raisebox{0.3cm}{
\begin{ytableau}
 1&\scalebox{0.8}{1,2}&2\\
 2\\
 3
\end{ytableau}}\,\right)\\
& \to 
\left(\,
\raisebox{0.3cm}
{\ytableaushort{123,2,3}}\,,\,
\raisebox{0.3cm}{
\begin{ytableau}
 1&\scalebox{0.8}{1,2}&\scalebox{0.8}{2,3}\\
 2\\
 3
\end{ytableau}}\,\right)
\end{align*}
and similarly we may compute the Hecke column insertion $\bDR:\CP_w\to \DR$ and get the result
\[
\left(\,
\raisebox{0.3cm}
{\ytableaushort{321,2,1}}\,,\,
\raisebox{0.3cm}{
\begin{ytableau}
 3&\scalebox{0.8}{3,2}&\scalebox{0.8}{2,1}\\
 2\\
 1
\end{ytableau}}\,\right).
\]
However, the row versions of these algorithms, which we denote $\bIR^{\mathrm{Hecke}}$ and $\bDS^{\mathrm{Hecke}}$, do not work: the recording tableaux can fail to be set-valued (or reverse set-valued). Explicitly,
if we apply $\bIR^{\mathrm{Hecke}}$ on the $(a,i)$ above, the recording tableau fails to be set-valued after four insertions:
\[\left(\,
 \raisebox{-0.1cm}
{\ytableaushort{1}}\,,\,
\raisebox{-0.1cm}
{\ytableaushort{3}}\,\right)\to
\left(\,
\raisebox{-0.1cm}
{\ytableaushort{13}}\,,\,
\raisebox{-0.1cm}
{\ytableaushort{33}}\,\right)\to
\left(\,
\raisebox{0.1cm}
{\ytableaushort{13,3}}\,,\,
\raisebox{0.1cm}
{\ytableaushort{33,2}}\,\right) \to
\left(\,
\raisebox{0.1cm}
{\ytableaushort{12,3}}\,,\,
\raisebox{0.1cm}{
\begin{ytableau}
 3&3\\
 \textcolor{red}{\scalebox{0.8}{2,2}}
\end{ytableau}}\,\right).
\]
For $\bDS^{\mathrm{Hecke}}$, we suffer from a similar pathology:
\[\left(\,
 \raisebox{-0.1cm}
{\ytableaushort{3}}\,,\,
\raisebox{-0.1cm}
{\ytableaushort{1}}\,\right)\to
\left(\,
\raisebox{-0.1cm}
{\ytableaushort{31}}\,,\,
\raisebox{-0.1cm}
{\ytableaushort{11}}\,\right)\to
\left(\,
\raisebox{0.1cm}
{\ytableaushort{31,1}}\,,\,
\raisebox{0.1cm}
{\ytableaushort{11,2}}\,\right) \to
\left(\,
\raisebox{0.1cm}
{\ytableaushort{32,1}}\,,\,
\raisebox{0.1cm}{
\begin{ytableau}
 1&1\\
 \scalebox{0.8}{\textcolor{red}{2,2}}
\end{ytableau}}\,\right).
\]
\end{example}

This paper introduces a new insertion algorithm $\Phi$ which gives an explicit weight-preserving 
bijection $\bDS: \CP_w \rightarrow \DS_w$.
Our algorithm is a row insertion which, like Hecke insertion, respects the Hecke equivalence relation $\Hequiv$. Our insertion possesses a different Pieri property than the one satisfied by Hecke row insertion; this is necessary to achieve set-valued recording tableaux.
Moreover, when restricted to $\CP_w^\Red$, our algorithm recovers EG row insertion.
A simple variation of our algorithm (reversing the total order on entries) gives a bijection 
$\bIR: \CP_w \rightarrow \IR_w$. Together with the variants of Hecke insertion, our insertion completes the picture in \S \ref{SS:general insertion}: we have produced the generalization of the four diagonal maps in Figure \ref{F:reduced bijections}.
Now the picture looks like Figure~\ref{F:bijections}.
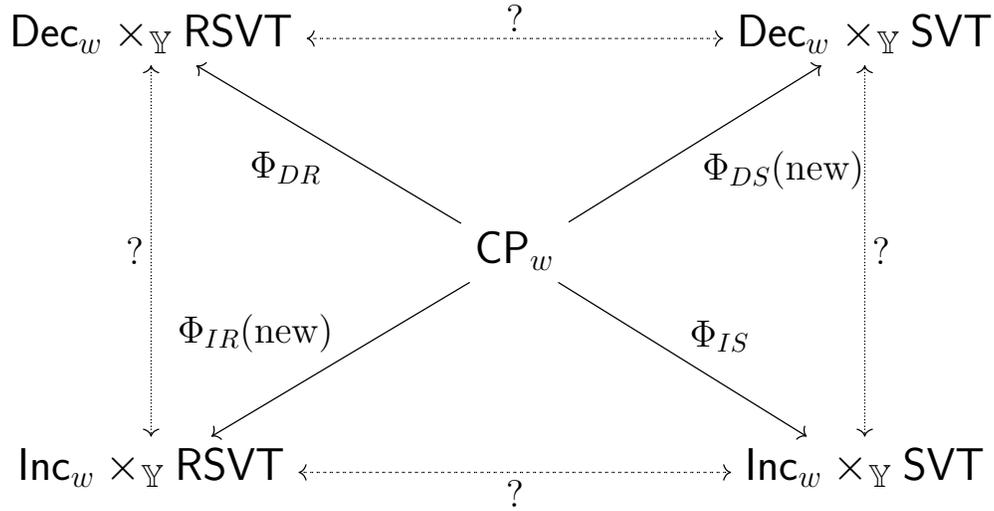
\begin{figure}
\[
\tikzcdset{every label/.append style = {font = \large}, nodes={font=\Large}}
\begin{tikzcd}
[row sep=5em, column sep = 5em]
\Large{\DR} \arrow[rr, "?",dotted] \arrow[dd, swap, "?",dotted]  &&
\DS \arrow[ll,dotted] \arrow[dd, "?",dotted]\\
& \CP_w \arrow[lu,"\bDR"]\arrow[ru,swap,"\bDS \textrm{(new)}"]\arrow[ld,swap,"\bIR\textrm{(new)}"]\arrow[rd,"\bIS"]\\
\IR \arrow[uu,dotted]\arrow[rr, dotted, swap, "?"] &&
\IS \arrow[ll, dotted] \arrow[uu,dotted]
\end{tikzcd}
\]
\caption{Non-reduced analogue of Figure~\ref{F:reduced bijections}}
\label{F:bijections}
\end{figure}

Our insertion generally gives a tableau of a different shape than row Hecke insertion.

\begin{example} Both our insertion 
and (the decreasing tableau version of) 
row Hecke insertion of the word $2421$ produce the same tableau $T$.
Let $P$ and $P'$ denote the tableau obtained when the number $3$ is inserted, for our insertion and row Hecke insertion respectively. $P$ and $P'$ have different shapes. See Example \ref{X:the forward insertion} for our insertion of $3$ into $T$.
\[
T = \begin{ytableau} 4&2&1\\2
\end{ytableau} \qquad
P = \begin{ytableau}
    4&3&1\\ 2&1 
\end{ytableau}\ne
P' = \begin{ytableau}
4&3&1\\
2
\end{ytableau}.
\]
\end{example}

\begin{remark} None of the coherence properties of 
Propositions \ref{P:EG left right} or \ref{P:EG inc dec} generalize to any of the 
bijections in the nonreduced setting. In general the four bijections produce 4 different 
groupings of compatible pairs for the various expansions of $G_w$. 
\end{remark}

In addition, we are not aware of 
any map that can be one of the four maps labeled by the question marks. 
Unlike the reduced case, these four maps 
have to change both the insertion tableaux and recording tableaux.
Even worse, as shown in the following examples,
these four maps are not shape-preserving. 
Therefore, doing an analogue of $\ev(\cdot)$
using the K-Bender-Knuth moves in~\cite{MPS} would not work
since the moves are shape-preserving.

\begin{example} Consider the following element $(a,i)\in \CP_{31524}$.
\begin{align*}
\begin{array}{|c||ccccc|} \hline 
i & 3&3&2&2&1 \\
a & 2&4&1&3&1 \\ \hline
\end{array}
\end{align*}
We first compute $\Phi_{DS}(a,i)$.
Starting with the empty tableau pair we insert the word $a$ starting with the rightmost entry using our row insertion. We obtain the following sequence of tableaux, where the ending box is green. 

\ytableausetup{boxsize=5.6mm,aligntableaux=top}
\[
\varnothing, \quad
\begin{ytableau}
 *(green!30) 1
\end{ytableau},\quad
\begin{ytableau}
  3 \\ *(green!30) 1 
\end{ytableau},\quad
\begin{ytableau}
  3& *(green!30) 1 \\ 1
\end{ytableau},\quad
\begin{ytableau}
  4& 1 \\ 3 \\ *(green!30) 1
\end{ytableau},\quad
\begin{ytableau}
  4&2 \\ 3&*(green!30) 1\\ 1 
\end{ytableau}
\]

The tableau pair $\Phi_{DS}(a,i)$ is
\ytableausetup{boxsize=5.8mm,aligntableaux=top}
\[
\begin{ytableau}
 4&2\\
 3&1\\
 1
\end{ytableau}\,,\,
\begin{ytableau}
 1&2\\
 2&3\\
 3
\end{ytableau}
\]
Under successive Hecke column insertions for increasing tableaux, we obtain 
\[
\varnothing,\quad 
\begin{ytableau} *(green!30) 1 \end{ytableau}, \quad
\begin{ytableau} 1 \\ *(green!30) 3 \end{ytableau}, \quad
\begin{ytableau} 1& *(green!30) 3 \\  3 \end{ytableau}, \quad
\begin{ytableau} 1& 3 \\ 3 \\ *(green!30) 4\end{ytableau},\quad
\begin{ytableau} 1& *(green!30)3 \\ 2 \\ 4\end{ytableau},\quad
\]
so $\Phi_{IS}(a,i)$ is the tableau pair
\[
\begin{ytableau} 1&3\\2\\4\end{ytableau},
\begin{ytableau} 1& \scalebox{0.8}{2,3} \\ 2 \\ 3 \end{ytableau}.
\]
This implies that the map on the right in Figure \ref{F:bijections} cannot be shape-preserving.
By simply reversing both words in $(a,i)$ and replacing
every number $j$ by $5 - j$, we obtain an example that 
implies the map on the left in Figure \ref{F:bijections} cannot be shape-preserving.
\end{example}

\begin{example} Consider the following element $(a,i)\in \CP_{24153}$.
\begin{align*}
\begin{array}{|c||ccccc|} \hline 
i & 5&4&4&4&2 \\
a & 3&1&2&4&2 \\ \hline
\end{array}
\end{align*}
We first compute $\Phi_{DS}(a,i)$.
Starting with the empty tableau pair we insert the word $a$ starting with the rightmost entry using our row insertion. We obtain the following sequence of tableaux, where the ending box is green. 

\ytableausetup{boxsize=5.6mm,aligntableaux=top}
\[
\varnothing, \quad
\begin{ytableau}
 *(green!30) 2
\end{ytableau},\quad
\begin{ytableau}
  4 \\ *(green!30) 2 
\end{ytableau},\quad
\begin{ytableau}
  4& *(green!30) 2 \\ 2
\end{ytableau},\quad
\begin{ytableau}
  4& 2 &*(green!30) 1 \\ 2
  \end{ytableau},\quad
\begin{ytableau}
  4&3&1 \\ 2& *(green!30)1 
\end{ytableau}
\]

The tableau pair $\Phi_{DS}(a,i)$ is
\ytableausetup{boxsize=5.8mm,aligntableaux=top}
\[
\begin{ytableau}
  4&3&1 \\ 2&1 
\end{ytableau}\,,\,
\begin{ytableau}
 2&4&4\\
 4&5
\end{ytableau}
\]
Under successive Hecke column insertions for decreasing tableaux starting with the left end, we obtain 
\[
\varnothing,\quad 
\begin{ytableau} *(green!30) 3 \end{ytableau}, \quad
\begin{ytableau} 3 \\ *(green!30)1 \end{ytableau}, \quad
\begin{ytableau} 3&*(green!30) 1 \\ 2 \end{ytableau}, \quad
\begin{ytableau} 4& 3&*(green!30) 1 \\ 2 \end{ytableau},\quad
\begin{ytableau} 4& 3&1 \\ *(green!30) 2\end{ytableau}
\]
so $\Phi_{DR}(a,i)$ is the tableau pair
\[
\begin{ytableau} 4& 3&1 \\ 2\end{ytableau},
\begin{ytableau} 5&4&4  \\ \scalebox{0.8}{4,2} \end{ytableau}.
\]
This implies that the map on the top in Figure \ref{F:bijections} cannot be shape-preserving.
By simply reversing both words in $(a,i)$ and replacing
every number $j$ by $6 - j$, we obtain an example that 
implies the map on the bottom in Figure \ref{F:bijections} cannot be shape-preserving.
\end{example}

\subsection{Restriction to bounded compatible pairs}

A compatible pair $(a_1 \cdots a_n, i_1 \cdots i_n )$ is 
{\em bounded} if $a_j \geq i_j$ for all $j \in [n]$.
Let $\CP_w^b$ be the set of bounded compatible pairs
in $\CP_w$.
Fomin and Kirillov~\cite{FK} showed that the generating 
function of $\CP_w$ is the Grothendieck polynomial,
which can be viewed as the non-symmetric refinement
of $G_w$.

Shimozono and Yu characterized
the image of $\CP_w^b$ under $\Phi_{DR}$ 
in~{\cite[Definition~4.1, Theorem~4.2]{SY}}. 
Shimozono and Yu used this description to
expand Grothendieck polynomials
into Lascoux polynomials positively. 
This expansion was first conjectured 
by Reiner and Yong~\cite{RY}.

In a follow-up work, 
Orelowitz and Yu~\cite{OY} characterized
the image of $\CP_w^b$ under the map $\Phi_{DS}$
introduced in this paper. 
Their description also leads to the 
Grothendieck-to-Lascoux expansion, as shown
in ~{\cite[Corollary~6.15]{OY}}.
Furthermore, they use the restriction of
$\Phi_{DS}$ on $\CP_w^b$
to expand the product of a Lascoux
polynomial and $G_w$ into Lascoux polynomials. 
See~{\cite[Remark. 3.4]{OY}} for a discussion
on why their arguments cannot work
using $\Phi_{DR}$.

\subsection{Related and future works}
We mention a few works in the literature where the usual Hecke row insertion algorithm of (Hecke) words were applied and studied. It would be interesting to investigate how our insertion algorithm behave in these contexts. We thank the anonymous referee for pointing out these relevant works. 

In \cite{THOMAS2011610}, Thomas and Yong studied the row insertion of a Hecke word through the lens of sampling algorithms for probability measures and proved a symmetry property of the insertion tableaux. In \cite{GUO2020105304}, Guo and Poznanovi\'c proved that the number of 0-1 filling of a stack polyomino subject to certain restrictions only depends on the set of row lengths, using the Hecke insertion algorithm as a main technical tool. Their work was extended by Bloom and Saracino \cite{BLOOM2024103898}. These works exploit the properties of the longest increasing/decreasing subsequences of a word being encoded in the insertion tableau. We leave the investigation of similar properties for our algorithm for future work.

\section{New reverse row insertion}
\subsection{Ejectable values in decreasing tableaux}
To define the new reverse insertion algorithm on decreasing tableaux, we require the notion of an {\em ejectable} value in a decreasing tableau. This is defined recursively.

In this article English notation is used for partitions and tableaux.
A tableau is \textit{decreasing} if its entries strictly decrease from left to right along each row and strictly decrease from top to bottom in each column.
For a decreasing tableau $P$ let $P_{> r}$ denote the decreasing tableau
obtained by removing the first $r$ rows of $P$. Let $P_{\geq r} = P_{>(r-1)}$.

\begin{definition} Let $P$ be a decreasing tableau. A value $x$ is {\em $P$-ejectable} if $x$ occurs in the first row of $P$ and either
$x-1$ is not in the first row of $P$, or $x-1$ is in the first row of $P$ and $x-1$ is $P_{>1}$-ejectable.
\end{definition}

\begin{example} The value $3$ is $P$-ejectable for the tableau $P$ depicted below.
\ytableausetup{aligntableaux=center}
\[
P = 
\begin{ytableau}
7 & 6 & *(cyan!30) 3 & *(green!30) 2 \\
5 & 2 & 1 \\
3 &  1
\end{ytableau}
\]
Since $3$ and $2$ both occur in the first row, $3$ is $P$-ejectable if and only if $2$ is $P_{>1}$-ejectable.
\[
P_{>1} = 
\begin{ytableau}
5 & *(cyan!30)2 & *(green!30)1 \\
3 & 1
\end{ytableau}
\]
Since $2$ and $1$ occur in the first row of $P_{>1}$, $2$ is $P_{>1}$-ejectable if and only if $1$ is $P_{>2}$-ejectable.
\[
P_{>2}=\begin{ytableau}
3 & *(cyan!30) 1
\end{ytableau}
\]
Since the first row of $P_{>2}$ has a $1$ but no $0$, $1$ is $P_{>2}$-ejectable. Hence $3$ is $P$-ejectable.

The value 7 is not $P$-ejectable because there is a 6 in the first row but not in the second.
\end{example}

\subsection{Ejectable values and Hecke equivalence}
\label{SS:Hecke}
The \textit{0-Hecke monoid} is the quotient of the free monoid of words on the alphabet $\Z_{>0}$ by the relations
\begin{align*}
  ii &\Hequiv i \\
  i(i+1)i &\Hequiv (i+1)i(i+1) \\
  ij&\Hequiv ji \qquad\text{for $|i-j|\ge 2$.}
\end{align*}
The minimum-length elements of each $\Hequiv$ class are the reduced words of some permutation $w$, giving a canonical bijection between the $\Hequiv$ classes and permutations of $\Z_{>0}$ moving finitely many elements. We denote by $[a]_H$ the permutation associated with the $\Hequiv$ class of the word $a$.

The row-reading word $\row(P)$ of a tableau $P$ is the word $\dotsm u^{(2)} u^{(1)}$ where $u^{(i)}$ is the word given by reading the $i$-th row of $P$ from left to right.

\begin{lemma}
\label{L:ejectable and Hecke equivalence}
Let $P$ be a decreasing tableau.
If $x$ is an ejectable entry of $P$ then $\row(P) \Hequiv \row(P) x$.
\end{lemma}
\begin{proof} This is proved by 
induction on the number of rows in $P$. 
Let $w$ be the decreasing word given by the first row of $P$ and let $R$ be the 
set of letters in $w$. By definition $\row(P) = \row(P_{> 1})w$.
It suffices to show that
\begin{align}\label{E:sufficient equivalence}
  \row(P_{>1})w\Hequiv \row(P_{>1})wx.
\end{align}
If $x \in R$ and $x-1 \notin R$
then $w \Hequiv w x$ and hence \eqref{E:sufficient equivalence} holds.
Otherwise $x, x-1 \in R$
and the $x-1$ is ejectable in $P_{>1}$.
By the inductive hypothesis, 
$\row(P_{> 1}) \Hequiv \row(P_{> 1}) (x-1)$.
In this case $(x-1) w \Hequiv w x$ and
$$
\row(P_{>1})wx \Hequiv \row(P_{>1})(x-1) w\Hequiv \row(P_{>1})w
$$
and again \eqref{E:sufficient equivalence} holds as required.
\end{proof}

\subsection{Bumping paths}

Let $D(\la)=\{(i,j)\in\Z_{>0}^2\mid \text{$i\ge1$, $j\le \la_i$} \}$ be the diagram of the partition $\la$. The elements of $D(\la)$ are called the \textit{cells} of $\la$ and have a matrix-style indexing: the cell $(i,j)$ is depicted as a box in the $i$-th row and $j$-th column.
For a partition $\la$, a $\la$-removable cell is one that is at the end of its row and bottom of its column. For a tableau $P$, a $P$-removable cell is a $\la$-removable cell where $\la$ is the shape of $P$.

\begin{definition} Let $(r, c)$ be a removable cell for the decreasing tableau $P$.
The (reverse) {\em bumping path} of $(r, c)$ in $P$
is the following sequence of numbers $m_r< m_{r-1}< \dots< m_1$ together with their positions in $P$.
Let $m_r$ be the value of $P$ in $(r,c)$. With the entry $m_{i+1}$ in row $i+1$ defined, 
let $m_i$ be the smallest number in row $i$
such that $m_{i+1} < m_i$.
\end{definition}

\begin{example} A decreasing tableau and the bumping path for its removable cell $(3,2)$ are pictured below.
\[
\begin{ytableau}
8&7&*(green!30) 6 \\
5&*(green!30)4&2 \\
3&*(green!30)2\\
1
\end{ytableau}
\]
\end{example}

In the example above, 
notice that the column index is weakly increasing,
as you go up in bumping path.
This is true in general. 
\begin{lemma}
\label{L: Bumping path weakly right}
Let $m_r< m_{r-1}< \dots< m_1$ 
be a bumping path in $P$.
For $r \geq j > i \geq 1$,
the $m_i$ in row $i$ of $P$
is weakly right of the $m_j$ in row $j$ of $P$.
\end{lemma}
\begin{proof}
We only need to prove this claim for $j = i+1$.
Let $y$ be the number immediately 
above the $m_{i+1}$ in row $i+1$ of $P$,
We have $y > m_{i+1}$.
Thus, $y \geq m_i$, 
The $m_i$ in row $i$ is weakly right of the 
$y$ in this row,
which implies our claim. 
\end{proof}

The element in the first row of any bumping path is ejectable.

\begin{lemma}\label{L:ejected entry is ejectable}
Let $m_r < \cdots < m_1$
be the bumping path of a removable cell of $P$.
Then $m_1$ is ejectable in $P$.
\end{lemma}
\begin{proof} The proof proceeds by induction on the number of rows in $P$. Let $R$ be first row of $P$. If $m_1-1 \notin R$
then $m_1$ is ejectable in $P$.
Otherwise $m_1, m_1 - 1 \in R$.
Since $m_1$ is the smallest in $R$
such that $m_1 > m_2$, it follows that $m_2 = m_1-1$.
It suffices to show that $m_1 - 1 = m_2$ is ejectable in $P_{>1}$.
This follows from the inductive hypothesis
since $m_r < \cdots < m_2$
is the bumping path of a removable cell in $P_{>1}$.
\end{proof}

\subsection{New reverse insertion}
The reverse insertion algorithm is a map $\Psi$
\[
(P,s,\alpha) \mapsto (P',m)
\]
where the input triple consists of a decreasing tableau $P$, a $P$-removable cell $s=(r,c)$, and $\alpha\in\{0,1\}$. 
The output pair consists of a decreasing tableau $P'$ and $m\in \Z_{>0}$ such that
\begin{align}\label{E:reverse insertion shape change}
    \shape(P') = \begin{cases} \shape(P) & \text{if $\alpha=0$} \\
    \shape(P) - \{s\} & \text{if $\alpha=1$.}
    \end{cases}
\end{align}

For conceptual clarity we precompute the bumping path in $P$ starting at $(r,c)$. For $1\le i\le r$ let $m_i$ denote the entry in the $i$-th row of the bumping path. The output value $m$ is by definition the value $m_1$ in the first row of the bumping path.

The output tableau $P'$ will only differ from $P$ along the bumping path. It is only necessary to specify whether each $m_i$ on the bumping path gets replaced, and if so, by what value.
Unlike most insertion algorithms, the replacement value might not come from the bumping path, but does come from the row below. The behavior on each row is determined iteratively by decreasing $i$ based on the values $m_i$ and $m_{i+1}$, the $i$-th row of $P$, the subtableau $P'_{>i}$, and a status indicator $\alpha_{i+1}\in \{0,1\}$. The $i$-th iteration updates the $i$-th row of $P$ (which becomes the $i$-th row of $P'$) and produces $\alpha_i\in \{0,1\}$.

Let $P'$ be a working tableau which is initialized to $P$.
In the initialization step, if $\alpha = 1$, remove from $P'$ the removable cell in row $r$ and its contents $m_r$ and set $\alpha_r=1$ and $i=r-1$.
If $\alpha=0$ set $m_{r+1} = 0$, $\alpha_{r+1}=0$ and $i=r$.

The algorithm does the following for $i = r, r-1, \cdots,2, 1$.
Let $R$ be the set consisting of numbers in row $i$ of the current tableau $P'$ (or equivalently $P$, since $P$ and $P'$ only differ under row $i$). By definition, $m_i\in R$.

There are several cases. We give each a nickname and mnemonic.
\begin{enumerate}
\item[$\bullet$] \textbf{Dummy (D)}: If $m_i - 1 \in R$ (which implies $m_{i+1}=m_i-1$) do not change the $i$-th row and set $\alpha_i = \alpha_{i+1}$.
\item[$\bullet$] \textbf{Direct Replacement (DR)}: Otherwise if $\alpha_{i+1}=1$ and $m_{i+1}\notin R$, replace $m_i$ by $m_{i+1}$ in row $i$ of $P'$ and set $\alpha_i=1$.
\end{enumerate}

Suppose neither of the two above cases hold. Find the smallest ejectable entry $x$ in $P'_{>i}$ such that $m_{i} > x > m_{i+1}$.

\begin{enumerate}
\item[$\bullet$] \textbf{Indirect Replacement (IR)}:
Suppose $x$ exists.
Replace $m_i$ by $x$ in row $i$ of $P'$ and set $\alpha_i=1$.
\item[$\bullet$] \textbf{No Replacement (NR)}: 
Suppose $x$ does not exist. Do not change the $i$-th row and set $\alpha_i=0$.
\end{enumerate}

\begin{example}
\label{E: reverse insertion}
In the following example, the input parameters are $s=(5,1)$ and $\alpha=0$. To initialize, set $(m_6,m_5,m_4,m_3,m_2,m_1)=(0,1,2,5,6,8)$, $i=5$, and $\alpha_6=0$. The shaded box in the $i$-th row indicates the value $m_i$. The label on the arrow leaving this tableau is the mnemonic for the case of $\Psi$.
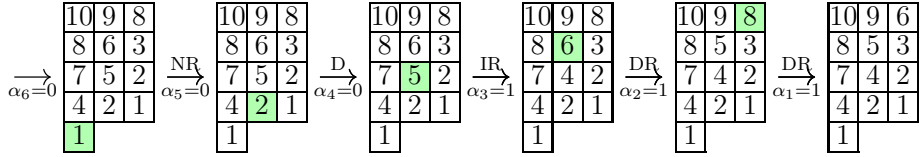
\[
\ytableausetup{boxsize=4.25mm}
\tikzcdset{scale cd/.style={every label/.append style={scale=#1},
    cells={nodes={scale=#1}}}}
\begin{tikzcd}[scale cd=.9,ampersand replacement=\&,/tikz/commutative diagrams/column sep= 5mm]
{}\arrow[r,swap,"\alpha_6=0"] \&
\begin{ytableau}
 10&9&8\\
 8&6&3 \\
 7&5&2\\
 4&2&1\\
 *(green!30) 1
\end{ytableau} \arrow[r,"\textrm{NR}"] \arrow[r,swap,"\alpha_5=0"]\&
\begin{ytableau}
 10&9&8\\
 8&6&3 \\
 7&5&2\\
 4&*(green!30) 2&1\\
 1
\end{ytableau} \arrow[r,"\textrm{D}"] \arrow[r,swap,"\alpha_4=0"] \& 
\begin{ytableau}
 10&9&8\\
 8&6&3 \\
 7&*(green!30) 5&2\\
 4&2&1\\
 1
\end{ytableau}
\arrow[r] \arrow[r,"\textrm{IR}"] \arrow[r,swap,"\alpha_3=1"] \&
\begin{ytableau}
 10&9&8\\
 8&*(green!30) 6&3 \\
 7&4&2\\
 4&2&1\\
 1
\end{ytableau}
\arrow[r] \arrow[r,"\textrm{DR}"] \arrow[r,swap,"\alpha_2=1"] \& 
\begin{ytableau}
 10&9&*(green!30) 8\\
 8&5&3 \\
 7&4&2\\
 4&2&1\\
 1
\end{ytableau}
\arrow[r] \arrow[r,"\textrm{DR}"] \arrow[r,swap,"\alpha_1=1"] \& 
\begin{ytableau}
 10&9&6\\
 8&5&3 \\
 7&4&2\\
 4&2&1\\
 1
\end{ytableau}
\end{tikzcd}
\]

\end{example}
\section{Properties of the reverse insertion}
In this section the reverse insertion map $\Psi$ is shown to be well-defined and 
some of its properties are established. 

\begin{lemma}
\label{L: Feint is ejectable} 
For $r \geq i \geq 1$, 
$\alpha_i=0$ if and only if $m_i$ is ejectable in $P'_{\ge i}$.       
\end{lemma}
\begin{proof} Note that after the $i$-th row is processed,
the subtableau $P'_{\ge i}$ remains the same thereafter: only bumping path entries in rows above may be changed.

The proof proceeds by descending induction on $i$.
For the initial step, if $\alpha = 1$, the algorithm sets $\alpha_r=1$ and $m_r$ gets removed and is therefore absent from the $r$-th row of $P'$. 
Thus, $m_r$ is not ejectable in $P_{\ge r}$.
If $\alpha = 0$, during the first iteration, $m_r$ is replaced in the $r$-th row of $P'$ if and only if $\alpha_{r}=1$. 
Hence our claim holds for $i=r$.

Now suppose the claim holds for row $i+1$.
In the Dummy case, $m_{i+1} = m_i - 1$ and $m_i$ is not replaced. Thus $m_i$ is ejectable in $P'_{\ge i}$
if and only if the entry $m_i-1=m_{i+1}$ is ejectable in $P'_{>i}$
(by definition of ejectable) if and only if $\alpha_{i+1}=0$ (by induction) if and only if $\alpha_i=0$ (since in the Dummy case $\alpha_i=\alpha_{i+1}$). 
Otherwise suppose the Dummy case does not hold.
Then $m_i$ and $m_i - 1$ cannot both live in row $i$ of $P'$.
Thus $m_i$ is ejectable in $P'_{\ge i}$ 
if and only if it is not removed from row $i$.
This happens only in the No Replacement case 
and $\alpha_i=0$ only in that case.
Thus our claim holds for row $i$ as required.
\end{proof}

\begin{theorem}
The reverse insertion is a well-defined map.
\end{theorem}
\begin{proof}
It must be shown that the output tableau $P'$ is a decreasing tableau. 
We show $P'$ is a decreasing tableau after each iteration of the algorithm.

During initialization, 
if $\alpha=0$, the iteration for $i=r$
either leaves $P'$ unchanged or replaces
$m_r$ in the removable cell $(r,c)$ by a smaller number. 
If $\alpha=1$, after the initialization step $P'$ is a decreasing tableau,
since it is obtained from a decreasing tableau by removing a corner entry.
In either case $P'$ is a decreasing tableau before the $i=r-1$ iteration.

Suppose $P'$ is a decreasing tableau 
after the iteration for row $i+1$. 
It is enough to check that $P'$ is still a decreasing tableau 
after the iteration for row $i$. 
In the Dummy or No Replacement cases there is nothing to check.
In the two remaining cases, 
the number $m_i$ is replaced by a smaller number. 
We need to make sure this smaller number 
is larger than all numbers
on its right and under it. 

Consider $P'$ before this iteration.
By the definition of a bumping path,
numbers on the right of the $m_i$ in row $i$
are at most $m_{i+1}$.
By Lemma~\ref{L: Bumping path weakly right},
numbers below this $m_i$
are also at most $m_{i+1}$.
Next, we consider the two cases.

\begin{enumerate}
\item[$\bullet$] Direct Replacement: 
$m_i$ is replaced by $m_{i+1}$.
We need to make sure $m_{i+1}$ 
is not on the right or under this $m_i$
in $P'$ before this iteration.
First, $m_{i+1}$ cannot be in row $i$
by the condition of this case.
Now assume toward contradiction that
$m_{i+1}$ is immediately below $m_i$. 
This part of $P'$ looks like
$$
\ytableausetup{boxsize=2.5em}
\begin{ytableau}
m_i & a\\
m_{i+1} & b
\end{ytableau}
\ytableausetup{boxsize=2em}
$$
Since $\alpha_{i+1}=1$, $m_{i+1}$ is not ejectable in $P'_{\ge i+1}$.
Thus, $b = m_{i+1} - 1$.
Then $a > m_{i+1} - 1$ and $m_{i+1} \geq a$,
so $a = m_{i+1}$. 
Since $m_{i+1}$ is not in row $i$ a contradiction is reached.
Thus, after replacing the $m_i$ by $m_{i+1}$,
$P'$ is still a decreasing tableau.
\item[$\bullet$] Indirect Replacement: $m_i$ is replaced by $x$.
We know $x > m_{i+1}$.
After replacing $m_i$ by $x$,
$P'$ is still a decreasing tableau,
since numbers on its right and under it are at most $m_{i+1}$. \qedhere
\end{enumerate}
\end{proof}

The reverse insertion respects Hecke equivalence.
\begin{lemma}
\label{L: Reverse insertion Hecke class}
Let $\Psi(P, (r,c), \alpha)=(P', m)$.
Then $\row(P) \Hequiv \row(P') m$.
\end{lemma}
\begin{proof}
Let $w$ be the decreasing word given by the first row of $P$ and let $R$ be the 
set of letters in $w$.
Define $R'$ and $w'$ similarly for $P'$.
Notice that $\row(P) = \row(P_{>1}) w$
and $\row(P') = \row(P_{>1}') w'$.
It suffices to show that 
\begin{align}\label{E:sufficient equivalence II}
\row(P_{>1}) w \Hequiv \row(P_{>1}') w' m.
\end{align}

The proof proceeds by induction on $r$,
the row index of the entry in the input of $\Psi$.
The base case is $r=1$.
In this case $\row(P_{>1}) = \row(P_{>1}')$.
If $\alpha = 1$,
then $w = w' m$.
Otherwise, in the first iteration,
the algorithm searches
for the smallest ejectable $x < m$ in $P_{>1}$.
If $x$ does not exist
then $w \Hequiv w m = w' m$.
Otherwise
$w'$ is obtained by 
changing $m$ in $w$ into $x$.
We see that \eqref{E:sufficient equivalence II} holds:
$$
\row(P_{>1})w 
\Hequiv \row(P_{>1}) x w
\Hequiv \row(P_{>1}) w' m.
$$

For the inductive step let $r > 1$.
Before the last iteration the algorithm behaves
as if doing $\Psi$ on $(P_{>1}, (r-1, c), \alpha)$.
By the definition of $\Psi$ the result is $(P_{>1}', m_2)$.
By the inductive hypothesis,
$\row(P_{>1}) \Hequiv \row(P_{>1}') m_2$.
It is enough to check 
$$\row(P_{>1}') m_2 w \Hequiv \row(P_{>1}') w' m. $$

Consider the first two cases of the last iteration.
\begin{itemize}
\item Dummy:
In this case, $m, m - 1 \in R$
and $m_2 = m - 1$.
We have $m_2 w \Hequiv w m = w' m$.
\item Direct Replacement:
In this case, $m - 1, m_2 \notin R$.
We know $w$ is obtained from $w$
by changing $m$ into $m_2$.
We have $m_2 w \Hequiv w' m$.
\end{itemize}
We may assume the above two cases do not hold. Then either
$m_2 \in R$ or $\alpha_2=0$ ($m_2$ is ejectable in $P'_{>1}$).
In either case we claim $\row(P_{>1}') m_2 w \Hequiv \row(P_{>1}')w$:
If $m_2 \in R$,
then $m_2 w\Hequiv w$
since $m_2 + 1$ is not in $w$ and $m_2$
is in $w$.
If the $m_2$ is ejectable in $P'_{>1}$
then $\row(P_{>1}')m_2 \Hequiv \row(P_{>1}')$
by Lemma~\ref{L:ejectable and Hecke equivalence}.

With this claim, it must be shown that
$$\row(P_{>1}') w \Hequiv \row(P_{>1}') w' m. $$
It must be verified that this holds 
in the remaining two cases:
\begin{itemize}
\item Indirect Replacement:
Since $x$ is ejectable in $P_{>1}'$,
$\row(P_{>1}') \Hequiv \row(P_{>1}') x$.
$w'$ is obtained by changing $m$ to $x$ in $w$.
Thus $x w \Hequiv w' m$.
\item No Replacement: We have $w\Hequiv wm=w'm$. \qedhere
\end{itemize}
\end{proof}

Our reverse row insertion
satisfies the following Pieri condition,
which is not satisfied by Hecke reverse
row insertion.

\begin{lemma}
\label{L: Pieri of reverse insertion} Let $\alpha,\alpha'\in \{0,1\}$, 
$P$ a decreasing tableau with removable corner $(r_1,c_1)$,
\[\Psi(P, (r_1,c_1), \alpha)=(P', m) \text{ and }
\Psi(P',(r_2,c_2), \alpha')=(P'',m')\] with $(r_2,c_2)$ a removable corner of $P'$ with $c_2 < c_1$. Then $m'>m$.
\end{lemma}
\begin{proof}
Let $m_{r_1}<\dotsm <m_1$ be the bumping path for $\Psi$ on $(P,(r_1,c_1),\alpha)$ and $n_{r_2} < \cdots < n_1$ the bumping path for $\Psi$ on $(P',(r_2, c_2),\alpha')$.
By definition $m_1=m$ and $n_1=m'$ so it is enough to show that $n_1>m_1$.

When $m_i$ is ejected from $P_{\geq i}'$,
$n_i$ is a number in the top row
of $P_{\geq i}'$.
Moreover, since $n_i$
is the last number
in a bumping path in $P_{\geq i}'$,
$n_i$ is ejectable in $P_{\geq i}'$.
We check $n_i > m_i$ for all $1\le i\le r_1$ by descending induction on $i$; in the case $\alpha=0$ the initial index is $i=r_1+1$.

For the base case consider the value of $\alpha$.
If $\alpha = 1$, 
$m_{r_1}$ is removed from row $r_1$
and ejected.
Clearly $n_{r_1} > m_{r_1}$.
If $\alpha = 0$,
$m_{r_1+1} = 0 < n_{r_1 + 1}$.

By induction we assume that 
$m_{i+1} < n_{i+1}$.
We consider the cases of the two reverse insertions when they process row $i$.
\begin{enumerate}
\item[$\bullet$] (Dummy case):
In this case, $m_i = m_{i+1} + 1$.
We have $n_i > n_{i+1} \geq m_i$.
\item[$\bullet$] (Direct Replacement case):
In this case, we replace $m_i$
by $m_{i+1}$ in row $i$.
Then $n_i$ is a number in row $i$ of $P'$, and since 
$n_i > n_{i+1}>m_{i+1}$ it must be to the left of $m_{i+1}$.
Thus $n_i$ is to the left of $m_i$
in row $i$ of $P$. We conclude that $n_i > m_i$.
\item[$\bullet$] (Indirect Replacement case):
In this case, we replace $m_i$
by $x$ on row $i$.
Since $n_{i+1} > m_{i+1}$
and $n_{i+1}$ is ejectable in $P_{> i}'$,
$n_{i+1}\geq x$ by the choice of $x$.
Thus $n_i$ is a number in row $i$ of $P'$, and since 
$n_i > n_{i+1} \ge x$, it must be to the left of $x$.
Similar to the previous case, 
$n_i > m_i$.
\item[$\bullet$] (No Replacement case):
In this case, 
there is no $x$ that is ejectable
in $P_{\geq i+1}'$ and $m_{i+1} < x < m_i$.
Since $n_{i+1}$ is ejectable in $P_{\geq i+1}'$,
$n_{i+1} \geq m_i$.
Thus $n_i > m_i$. \qedhere
\end{enumerate}
\end{proof}

\begin{example}\label{X:the example in reverse}
Let $P$ be the following decreasing tableau:
$$
\ytableausetup{boxsize=5mm}
\begin{ytableau}
4 & 3 & 1\\
2 & 1
\end{ytableau}.
$$
Invoke the reverse insertion with
input $(r_1,c_1) = (2,2)$ and $\alpha_1 = 1$.
We obtain $(P', m)$
where $m_1 = 3$ and 
$P'$ is the following decreasing tableau
$$
\begin{ytableau}
4 & 2 & 1\\
2
\end{ytableau}.
$$
If we invoke the reverse insertion on $P'$ with
input $(r_2,c_2) = (2,1)$, the output number will be $m' = 4$. 
This aligns with Lemma~\ref{L: Pieri of reverse insertion}
since $c_2 < c_1$ and $m' > m$.
\end{example}

The reverse insertion algorithm 
is a generalization of EG reverse insertion.

\begin{lemma}
\label{L: Reverse insertion generalizes EG}
Let $P$ be a decreasing tableau
such that $\row(P)$ is reduced. 
Let $\Psi(P, (r,c), 1)=(P', m)$.
Then we also get $(P', m)$
if we apply EG reverse row insertion
at $(r, c)$ in $P$.
\end{lemma}
\begin{proof}
Since the Dummy and Direct Replacement cases agree with 
EG reverse insertion, it is enough to show that during each iteration, 
one of these cases must apply.

For $\alpha=1$ the initial step agrees with reverse EG insertion.
By induction we assume $\alpha_{i+1}=1$. We will assume 
the iteration for row $i$ is not in the
Dummy nor the Direct Replacement cases and reach a contradiction.
Let $R$ be the $i$-th row of $P$.
We assume $m_i-1\not\in R$ and $m_{i+1}\in R$.
By the minimality of $m_i$, $m_{i+1} + 1\not\in R$.
Let $w$ be the row word of the first $i$ rows of $P$.
We have $m_{i+1} w \Hequiv w$.
Then notice that
$$
\row(P) = \row(P_{>i}) w
\Hequiv \row(P'_{>i}) m_{i+1} w
\Hequiv \row(P'_{>i}) w.
$$
Then $\row(P)$ is not reduced and we obtain the required
contradiction.
\end{proof}

\section{The insertion}
This section gives a direct description of the inverse of $\Psi$, an
insertion algorithm $\Phi$ which ``inserts $m$ into $P$":
\[
(P,m) \mapsto (P', s, \alpha)
\]
where the input pair consists of a decreasing tableau $P$ and $m\in \Z_{>0}$, and the output triple consists of a decreasing tableau $P'$, a removable cell $s=(r,c)$ of $P'$, and $\alpha\in\{0,1\}$ such that the following holds:
\begin{align}\label{E:insertion shape change}
    \shape(P') = \begin{cases} \shape(P) & \text{if $\alpha=0$} \\
    \shape(P) \cup \{s\} & \text{if $\alpha=1$.}
    \end{cases}
\end{align}

The working tableau $P'$ has initial value $P$. 
The $i$-th iteration consists of an insertion of a number $N\in \Z_{>0}$ into  $P'_{\geq i}$. At this point $P'_{\ge i}=P_{\ge i}$; only values in rows before the $i$-th have been changed. Let $R$ be the set consisting of numbers
in row $i$ of $P$. Find the largest $n_1 \in R$
such that $n_1 \leq N$.

\begin{enumerate}
\item[$\bullet$] \textbf{Terminating case 1 (T1)}: If $n_1$ does not exist,
put $N$ at the end of row $i$ in $P'$ and
terminate the algorithm.
The output $P'$ is the current tableau.
The output $(r,c)$ is the coordinate
of this newly added $N$. Set $\alpha=1$.
\end{enumerate}

Otherwise $n_1$ exists. 
Change the $n_1$ in row $i$ of $P'$ into $N$.

\begin{enumerate}
\item[$\bullet$] \textbf{Dummy case (D)}:
If $n_1 = N$ and $N-1 \in R$: 
insert $N-1$ into $P_{> i}'$.

\item[$\bullet$] \textbf{Direct Replacement case (DR)}:
If $n_1 < N$ and $n_1$ is not ejectable in $P_{> i}$:
insert $n_1$ into $P'_{> i}$.
\end{enumerate}

Otherwise assume none of the above cases hold.
Let $n_2$ be the number to the right 
of $n_1$ in row $i$ of $P$,
or $n_2 = 0$ if $n_1$ is the rightmost
number in this row. 
Find the largest ejectable $y$ in $P_{>i}$
such that $n_1 > y > n_2$.

\begin{enumerate}
\item[$\bullet$] 
\textbf{Indirect Replacement case 1 (IR1)}:
If $y$ exists: insert $y$ into $P_{>i}'$.
\item[$\bullet$] 
\textbf{Indirect Replacement case 2 (IR2)}:
If $y$ does not exist and $n_2>0$: 
insert $n_2$ into $P_{>i}'$.
\item[$\bullet$] 
\textbf{Terminating case 2 (T2)}:
If $y$ does not exist and $n_2=0$:
terminate the algorithm.
The output $P'$ is the current tableau.
The output $(r,c)$ is the coordinate
of this $N$ in row $i$ of $P'$. 

Set $\alpha=0$.
\end{enumerate}

\begin{example}
In the following example, we let $P$ be the leftmost tableau and insert $m=8$ into $P$. The output is the rightmost tableau $P'$, $s=(5,1)$, and $\alpha=0$. The unshaded part of each tableau is the part being considered by the insertion in each step.
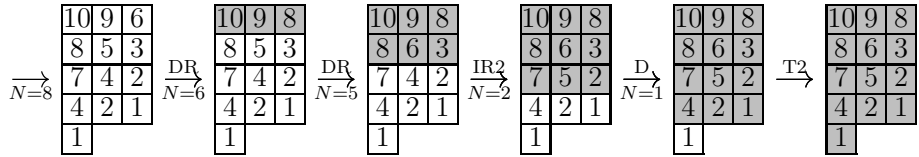
\[
\ytableausetup{boxsize=4.25mm}
\tikzcdset{scale cd/.style={every label/.append style={scale=#1},
    cells={nodes={scale=#1}}}}
\begin{tikzcd}[scale cd=.9,ampersand replacement=\&,/tikz/commutative diagrams/column sep=5mm]
{}\arrow[r,swap,"\substack{N=8}"] \&
\begin{ytableau}
 10&9&6\\
 8&5&3 \\
 7&4&2\\
 4&2&1\\
 1
\end{ytableau} \arrow[r,"\textrm{DR}"] \arrow[r,swap,"\substack{N=6}"]\&
\begin{ytableau}
 *(lightgray)10&*(lightgray)9&*(lightgray)8\\
 8&5&3 \\
 7&4&2\\
 4&2&1\\
 1
\end{ytableau} \arrow[r,"\textrm{DR}"] \arrow[r,swap,"\substack{N=5}"] \& 
\begin{ytableau}
  *(lightgray)10& *(lightgray)9& *(lightgray)8\\
  *(lightgray)8& *(lightgray)6& *(lightgray)3 \\
 7&4&2\\
 4&2&1\\
 1
\end{ytableau}
\arrow[r] \arrow[r,"\textrm{IR2}"] \arrow[r,swap,"\substack{N=2}"] \&
\begin{ytableau}
 *(lightgray)10&*(lightgray)9&*(lightgray)8\\
 *(lightgray)8&*(lightgray)6&*(lightgray)3 \\
 *(lightgray)7&*(lightgray)5&*(lightgray)2\\
 4&2&1\\
 1
\end{ytableau}
\arrow[r] \arrow[r,"\textrm{D}"] \arrow[r,swap,"N=1"] \& 
\begin{ytableau}
 *(lightgray)10&*(lightgray)9&*(lightgray)8\\
 *(lightgray)8&*(lightgray)6&*(lightgray)3 \\
 *(lightgray)7&*(lightgray)5&*(lightgray)2\\
 *(lightgray)4&*(lightgray)2&*(lightgray)1\\
 1
\end{ytableau}
\arrow[r] \arrow[r,"\textrm{T2}"]  \& 
\begin{ytableau}
 *(lightgray)10&*(lightgray)9&*(lightgray)8\\
 *(lightgray)8&*(lightgray)6&*(lightgray)3 \\
 *(lightgray)7&*(lightgray)5&*(lightgray)2\\
 *(lightgray)4&*(lightgray)2&*(lightgray)1\\
 *(lightgray)1
\end{ytableau}
\end{tikzcd}
\]
\end{example}

\begin{example}
In the following example, we insert $m=5$ into the leftmost tableau $P$. The output is the rightmost tableau $P'$, $s=(3,2)$, and $\alpha=1$.
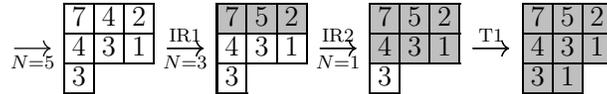
\[
\ytableausetup{boxsize=4.25mm}
\tikzcdset{scale cd/.style={every label/.append style={scale=#1},
    cells={nodes={scale=#1}}}}
\begin{tikzcd}[scale cd=.9,ampersand replacement=\&,/tikz/commutative diagrams/column sep= 5mm]
{}\arrow[r,swap,"\substack{N=5}"] \&
\begin{ytableau}
 7&4&2\\
 4&3&1\\
 3
\end{ytableau} \arrow[r,"\textrm{IR1}"] \arrow[r,swap,"\substack{N=3}"]\&
\begin{ytableau}
 *(lightgray)7&*(lightgray)5&*(lightgray)2\\
 4&3&1\\
 3
\end{ytableau} \arrow[r,"\textrm{IR2}"] \arrow[r,swap,"\substack{N=1}"] \& 
\begin{ytableau}
  *(lightgray)7&*(lightgray)5&*(lightgray)2\\
 *(lightgray)4&*(lightgray)3&*(lightgray)1\\
 3
\end{ytableau}
\arrow[r] \arrow[r,"\textrm{T1}"]  \&
\begin{ytableau}
*(lightgray)7&*(lightgray)5&*(lightgray)2\\
 *(lightgray)4&*(lightgray)3&*(lightgray)1\\
 *(lightgray)3 &*(lightgray)1
\end{ytableau}
{}
\end{tikzcd}
\]

\end{example}
\section{Properties of the insertion}
In this section the well-definedness of the insertion algorithm $\Phi$ is established 
and some of its properties are studied. 

\begin{lemma}
\label{L: box}
Consider an iteration of $\Phi$
in which $N$ is being inserted into $P_{\geq i}'$ 
in Indirect Replacement case 2. 
Consider the value $n_1$ in row $i$ of $P$.
If there is a number below this $n_1$,
it must be at most $n_2$.
\end{lemma}
\begin{proof}
Let $t_1$ be the number below this $n_1$.
This part of $P$ looks like
$$
\begin{ytableau}
n_1 & n_2\\
t_1 & t_2
\end{ytableau}
$$

Now assume toward contradiction that 
$t_1 > n_2$.
The number $t_2$ either does not exist
or we have $t_2 < n_2 \leq t_1 - 1$.
In either case, 
$t_1$ is ejectable in $P_{> i}'$.
By $n_1 > t_1 > n_2$,
we should go to Indirect Replacement case 1.
Contradiction.
\end{proof}

\begin{lemma}
The insertion algorithm is well-defined.
\end{lemma}
\begin{proof} The algorithm initializes the working tableau to equal $P$ which is decreasing. To show the output tableau is decreasing it suffices to assume that before any particular iteration the working tableau is decreasing and show that after that iteration, the resulting tableau is decreasing.

Let $P'$ be the working tableau at the beginning of the current iteration, in which $N$ is being inserted into $P'_{\geq i}$. Let $P''$ be the working tableau after this iteration. During the iteration, in row $i$ the number
$n_1$ is replaced by $N$ or $N$ is appended at the end; 
let $(i,j')$ be the position of this $N$.
After this iteration, 
the row will clearly be strictly decreasing. 
We may assume $i>1$ and must 
show that there is a number $M$ in position $(i-1,j')$ of $P''$ and it satisfies $M > N$.

If $n_1 = N$, then we are done since this iteration does not
change the working tableau at all.
We assume $n_1 < N$, 
so the previous iteration is not in the Indirect Replacement case 1.
Consequently, $N$ is in row $i - 1$ of $P$, say at $(i - 1, j)$.
We have $j' \le j$ by the choice of $n_1$. 
In particular there is a number $M$ in position $(i-1,j')$ of $P''$.
It remains to show that $M>N$.

If $j'<j$ then we obtain the required inequality $M>N$ since the $(i-1)$-th row was strictly decreasing before the previous iteration. So we may assume $j'=j$.

We consider the cases of the previous iteration:
\begin{itemize}
\item Dummy case. $N+1$ and $N$ are in row $i - 1$ of $P$.
Below this $N+1$, we have a number at most $N$, so $j' \leq j - 1$, contradiction.
\item Direct Replacement case.
During the previous iteration, 
the $N$ in cell $(i-1, j)$ is replaced by a larger number.
Thus, there is an $M > N$ at $(i - 1, j)$ of $P''$.

\item Indirect Replacement case 2. By Lemma~\ref{L: box}, 
the $n_1$ is in the first $j-1$ columns, a contradiction.
\end{itemize}
Since the row $i-1$ iteration was not terminal and we ruled out Indirect Replacement case 1, all cases are covered.
\end{proof} 
\begin{theorem}\label{T: inverse bijections}
$\Phi$ and $\Psi$ are mutually inverse functions.
\end{theorem}

We prove Theorem 5.3 directly
by splitting the proof into two lemmas.
\begin{lemma} Let $\Psi(P, (r,c), \alpha)= (P', m)$.
Then $\Phi(P', m)=(P, (r,c), \alpha)$.
\end{lemma}
\begin{proof}
Let $R$ (resp. $R'$) consist of the numbers 
in row 1 of $P$ (resp. $P'$).
The proof proceeds by induction on $r$.

For the base case,
assume $r = 1$.
If $\alpha = 1$, 
$m = \min(R)$ and $R' = R - \{ m\}$.
When we insert $m$ into $P'$,
the first iteration is 
in Terminating case 1. 
We will just append $m$ 
at the end of row 1 
and terminate at this cell.
If $\alpha = 0$,
we study the cases of the only iteration
in the reverse insertion:
\begin{enumerate}
\item[$\bullet$] 
Indirect Replacement case: 
In this case, $R' = R - \{m\} \sqcup \{x\}$ where
$x < m$ and $x$ is the smallest number in row 2 of $P$.
When we insert $m$ into $P'$,
it sets $n_1 = x$.
Since $n_1$ is ejectable in $P_{>1}'$,
it does not go to the first 3 cases. 
Then we have $n_2 = 0$.
There are no ejectable numbers 
in $P_{> 1}'$
between $n_2$ and $n_1$.
Thus, it goes to the Terminating case 2.
It replaces $x$ by $m$
and ends at this cell with $\alpha = 0$.
\item[$\bullet$] 
No Replacement case: 
In this case, $R = R'$
and there are no ejectable numbers in
$P_{> 1}'$ that are less than $m$.
When we insert $m$ into $P'$,
it sets $n_1 = m$ and $n_2 = 0$.
Thus, it goes to Terminating case 2.
It replaces $m$ by $m$
and ends at this cell with $\alpha = 0$.
\end{enumerate}

Now assume $r > 1$.
Consider the reverse insertion.
Before the last iteration,
a number $m_2 > 0$ is ejected 
from $P_{>1}'$.
During the last iteration,
it changes at most one number in $R$
and get $R'$.
Then it ejects $m$. 
By induction it suffices to show that 
when $m$ is inserted into $P'$, the 
first iteration of insertion changes $R'$ back to $R$
and inserts $m_2$ into $P_{> 1}'$.
Let us do a case study on the last iteration of the reverse insertion.
\begin{itemize}
\item Dummy case: 
$m_2 = m - 1$ and $m_2, m \in R$.
The algorithm fixes row 1 of $P$ so $R = R'$.
The first iteration of insertion
goes to the Dummy case.
It fixes row 1 and inserts $m - 1 = m_2$
into $P_{> 1}'$.

\item Direct Replacement case: 
$m_2$ was ejected with $\alpha_2=1$.
Thus $m_2$ is not ejectable in $P_{> 1}'$.
The algorithm replaces $m$ by $m_2$.
The first iteration of insertion
sets $n_1 = m_2$.
It goes to the Direct Replacement case:
$m_2$ is replaced by $m$ and $m_2$ is inserted.

\item Indirect Replacement case:
$m$ is replaced by $x$ which is ejectable in $P_{> 1}'$.
By the choice of $x$ there are no ejectable numbers in $P'_{> 1}$
between $m_2$ and $x$.
The first iteration of insertion
sets $n_1 = x$ and replaces it by $m$. 
It will not go to the first 3 cases.
Since $m$ is the smallest number in $R$
that is larger than $m_2$,
$m_2 \geq n_2$. If $m_2 > n_2$ then $\alpha_2=0$.
Thus $m_2$ is ejectable in $P_{> 1}'$.
During the subsequent insertion, 
the algorithm sets $y = m_2$ and inserts $m_2$.
Now assume $m_2 = n_2$.
Then the first iteration of the insertion
cannot find such a $y$. It inserts $n_2 = m_2$.

\item No Replacement case:
$R = R'$.
There are no ejectable numbers in $P'_{>1}$
between $m_2$ and $m$.
During the insertion,
the algorithm sets $n_1 = m$
and will not go to the first three cases. The proof 
proceeds as in the Indirect Replacement case. \qedhere
\end{itemize}
\end{proof}

\begin{lemma}
Let $\Phi(P, m)= (P',(r,c), \alpha)$.
Then $\Psi(P', (r,c), \alpha)=(P, m)$.
\end{lemma}
\begin{proof}
Let $R$ (resp. $R'$) consist of the numbers 
in row 1 of $P$ (resp. $P'$).
The proof proceeds by induction on $r$.

In the base case $r=1$, the insertion has only one iteration.
If $\alpha = 1$, this iteration is in Terminating case 1.
It appends $m$ at the end of row 1.
During the subsequent reverse insertion,
$m$ will be removed from row 1 and ejected.
Now assume $\alpha = 0$.
If $n_1 = m$,
then the insertion leaves row 1 unchanged.
There are no ejectable numbers in $P_{>1}$
that are less than $m$.
During the reverse insertion,
the sole iteration goes to 
the No Replacement case:
row 1 is unchanged and $m$ is ejected.
If $n_1 < m$,
then the insertion replaces $n_1$ by $m$.
Since it is in the Terminating case 2,
$n_1$ is smallest ejectable number in $P_{> 1}$.
During the reverse insertion,
the only iteration goes to 
the Indirect Replacement case:
the $m$ is changed to $n_1$ and $m$ is ejected.
 
Now assume $r > 1$.
Consider the insertion.
During the first iteration,
it changes $n_1$ into $m$ in row 1.
Then it inserts a number into $P_{>1}'$.
Let $z$ be that number.
Now consider the reverse insertion.
By our inductive hypothesis,
before the last iteration,
$z$ is ejected under row 1 of the tableau.
Moreover, currently the tableau below
row 1 is identical to $P_{> 1}$.
We need to make sure the last iteration
changes $R'$ back to $R$
and ejects $m$.
Let us do a case study
on the first iteration of the insertion.

\begin{enumerate}
\item[$\bullet$] Dummy case:
$m, m-1 \in R$, $R = R'$, 
and $z = m-1$.
The last iteration of the reverse insertion
goes to the Dummy case:
it fixes the first row and ejects $m$.
\item[$\bullet$] Direct replacement case:
$n_1$ is changed to $m$ and $z = n_1 < m$.
Moreover $z$ is not ejectable in $P_{> 1}$.
When $z$ is ejected from $P_{> 1}$, $\alpha_2=1$.
The last iteration of the reverse insertion
goes to the Direct Replacement case:
It changes $m$ into $n_1$ and ejects $m$.

\item[$\bullet$] 
Indirect Replacement case 1:
$n_1$ is changed to $m$ and $z = y$.
$y$ is the largest ejectable
number in $P_{> 1}$ less than $n_1$.
Moreover $y > n_2$.
Consider the last iteration of
the reverse insertion.
Before this iteration, by induction and Lemma \ref{L: Feint is ejectable}
$y$ is ejected from $P_{> 1}$ with $\alpha_2=0$.
Then it sets $m_1 = m$.
It looks for $x$,
which is the smallest ejectable number
in $P_{> 1}$ between $y$ and $m$.
If $n_1 = m$,
then it goes to the No Replacement case:
Row 1 is fixed and $m$ is ejected. 
If $n_1 < m$ then $n_1$ must be ejectable in $P_{> 1}$ and $x=n_1$.
It goes to the Indirect Replacement case:
$m$ is replaced by $n_1$ and $m$ is ejected.

\item[$\bullet$] 
Indirect Replacement case 2:
$n_1$ is replaced by $m$
and $z = n_2 > 0$.
There are no ejectable numbers 
between $n_2$ and $n_1$ in $P_{> 1}$.
Consider the last iteration of
the reverse insertion.
It sets $m_1 = m$.
Since $n_2$ is already in row 1,
it must go to the last two cases.
If $n_1 = m$,
then it goes to the No Replacement case:
The first row is fixed and $m$ is ejected.
If $n_1 < m$,
then $n_1$ must be ejectable in $P_{> 1}$.
It goes to the Indirect Replacement case: 
$m$ is replaced by $n_1$ and $m$ is ejected.
\qedhere
\end{enumerate}
\end{proof}

Our insertion satisfies a Pieri property.

\begin{lemma}
\label{L: Pieri of insertion}
Let 
$\Phi(P, m) = (P', (r_1, c_1), \alpha)$
and $\Phi(P',m')=(P'',(r_2,c_2),\alpha')$.
If $m' < m$, then $c_1 < c_2$.
\end{lemma}

\begin{proof}
Let $m_{r_1} < \cdots < m_1$
be the bumping path of $(r_1, c_1)$ in $P'$. 
By the definition of $\Psi$ on $(P',(r_1,c_1),\alpha)$, 
the output value is $m_1$. 
Since $\Psi$ is inverse to $\Phi$,
$\Psi(P',(r_1,c_1),\alpha)=(P,m)$. Thus $m_1=m>m'$. 
If $\Phi$ on $(P',m')$ ends in the first iteration,
we are done. 
Otherwise, after this iteration,
another number is inserted into $P_{> 1}'$.
It is enough to ensure that this number is
smaller than $m_2$.
During this iteration,
$\Phi$ finds a number $n_1$
in row 1 of $P'$.
We have $n_1 \leq m' < m_1$.
Since $m_{r_1} < \cdots < m_1$
is a bumping path, $m_2 \geq n_1$.
Now consider the case of the first iteration. 
\begin{itemize}
\item Dummy case: The number $m' - 1$ is inserted into $P_{> 1}'$.
We have $m_2 \geq n_1 = m' > m' - 1$. 
\item Direct Replacement case:
The number $n_1$ is inserted into $P_{> 1}'$.
Notice that $m_2$ is ejectable in $P_{> 1}'$
since it is the end of a bumping path
in $P_{> 1}'$.
However, 
$n_1$ is not ejectable in $P_{> 1}'$
by the condition of this case. Thus $m_2\ne n_1$. Since $m_2\ge n_1$ we deduce that $m_2 > n_1$.
\item Indirect Replacement case 1:
The number $y$ is inserted into $P_{> 1}'$.
We have $m_2 \geq n_1 > y$.
\item Indirect Replacement case 2:
The number $n_2$ is inserted into $P_{> 1}'$.
We have $m_2 \geq n_1 > n_2$. \qedhere
\end{itemize}
\end{proof}

\begin{example}\label{X:the forward insertion}
Let $P$ be the following decreasing tableau:
$$
\ytableausetup{boxsize=5mm}
\begin{ytableau}
4 & 2 & 1\\
\end{ytableau}.
$$
After inserting the number $m = 4$ to $P$,
we obtain $(P', (r_1, c_1), \alpha)$
where $(r_1, c_1) = (2,1)$, $\alpha = 1$
and $P'$ is 
$$
\begin{ytableau}
4 & 2 & 1\\
*(green!30)2
\end{ytableau}.
$$

Next, insert $m' = 3$ to $P'$
and get $(P'', (r_2, c_2), \alpha')$
where $(r_2, c_2) = (2,2)$, $\alpha = 1$
and $P''$ is 
$$
\begin{ytableau}
4 & 3 & 1\\
2 & *(green!30)1
\end{ytableau}.
$$
This aligns with Lemma~\ref{L: Pieri of insertion}
since $c_1 < c_2$ and $m' < m$.
\end{example}

To summarize, our new reverse insertion
satisfies the following Pieri property. 

\begin{theorem}
\label{T: Pieri}
Let $P$ be a decreasing tableau.
Apply successive reverse insertions
\begin{align*}
    \Psi(P, (r_1,c_1),\alpha)&= (P',m) \\
    \Psi(P',(r_2,c_2),\alpha')&= (P'',m')
\end{align*}
Then $c_2 < c_1$ if and only if $m' > m$.
\end{theorem}
\begin{proof}
Follows from Lemma~\ref{L: Pieri of reverse insertion},  Lemma~\ref{L: Pieri of insertion}, and Theorem~\ref{T: inverse bijections}.
\end{proof}

The following is an equivalent restatement for insertion.

\begin{corollary}
\label{C:forward Pieri}
Let $P$ be a decreasing tableau and $m,m'\in \Z_{>0}$. Applying successive insertions
$\Phi(P,m)=(P',(r_1,c_1),\alpha)$ and $\Phi(P',m') = (P'',(r_2,c_2),\alpha')$, we have
 $m>m'$ if and only if $c_1<c_2$.
\end{corollary}

Given a compatible pair $(a,i)$ and starting with the empty tableau pair, use $\Phi$ to insert $a_1$, then $a_2$, and so on, recording the insertion of $a_k$ by $i_k$, producing a tableau pair $(P,Q)$ where $P$ is decreasing. Denote this map by $\PhiRSK(a,i)=(P,Q)$.

\begin{corollary} \label{C:RSK} $\PhiRSK$ is a weight-preserving bijection $\CP_w\to \DS$.
\end{corollary}
\begin{proof}
Corollary \ref{C:forward Pieri} implies that $Q$ is set-valued. Moreover it also implies that the process can be inverted: if $(P,Q)\in \DS$, then using $\Psi$ at the sequence of removable boxes given by the entries of $Q$, one recovers $(a,i)\in\CP_w$.
\end{proof}

\begin{remark} Reversing comparison of values, one obtains a weight-preserving bijection $\PhiRSK^{\Inc}:\CP_w\cong \IR$.
\end{remark}


%
%

\bibliographystyle{alphaurl}
\bibliography{ct-sample}

\end{document}